\documentclass[12pt]{article}

\usepackage{amssymb}
\usepackage{amsmath}
\usepackage{multicol}
\usepackage{url}
\usepackage{xcolor}
\usepackage[normalem]{ulem}

\usepackage{amsthm}
\theoremstyle{definition}
\newtheorem{defn}{Definition}[section]
\newtheorem{ex}[defn]{Example}
\newtheorem{rem}[defn]{Remark}
\theoremstyle{plain}
\newtheorem{lemma}[defn]{Lemma}

\newtheorem{thm}[defn]{Theorem}

\usepackage{hyperref}

\newcommand{\B}{{\mathcal B}}
\newcommand{\E}{{\mathcal E}}
\newcommand{\F}{{\mathbb F}}
\DeclareMathOperator\SRG{SRG}
\DeclareMathOperator\RG{RG}
\DeclareMathOperator{\PG}{PG}
\newcommand{\gs}[3]{\genfrac{[}{]}{0pt}{0}{#1}{#2}_{#3}}
\newcommand{\gstxt}[3]{\genfrac{[}{]}{0pt}{1}{#1}{#2}_{#3}}
\newcommand{\comments}[1]{}

\begin{document}
\newcounter{savecntr}
\newcounter{restorecntr}

\title{$q$-Analogs of strongly regular graphs}
\author{Michael Braun (\url{michael.braun@h-da.de})\footnote{Faculty of Computer Science, University of Applied Sciences, Darmstadt, Germany}\\
Dean Crnkovi\'c \footnote{Corresponding author} (\url{deanc@math.uniri.hr})\setcounter{savecntr}{\value{footnote}}\footnote{Faculty of Mathematics, University of Rijeka}\\
Maarten De Boeck (\url{mdeboeck@memphis.edu})\footnote{Department of Mathematical Sciences, University of Memphis}\\
Vedrana Mikuli\'c Crnkovi\'c (\url{vmikulic@math.uniri.hr})\setcounter{restorecntr}{\value{footnote}}\setcounter{footnote}{\value{savecntr}}\footnotemark
\\
and\\
Andrea \v{S}vob (\url{asvob@math.uniri.hr})\setcounter{footnote}{\value{savecntr}}\footnotemark\\
}

\maketitle

\begin{abstract}
We introduce the notion of $q$-analogs of strongly regular graphs and give several examples of such structures. We prove a necessary condition on the parameters, show the connection to designs over finite fields, and present a classification.
\end{abstract}

\vspace*{0.5cm}


\noindent {\bf AMS classification numbers:} 05E30, 05B05.

\section{Introduction}\label{sec:intro}

In graph theory, Bose \cite{Bos63} introduced the notion of a strongly regular graph with parameters $(v,k,\lambda,\mu)$ which is a regular graph on $v$ vertices with degree $k$ such that two adjacent vertices have $\lambda$ common neighbors and two non-adjacent vertices have $\mu$ common neighbors. Classical examples of strongly regular graphs are the pentagon graph with parameters $(5,2,0,1)$ and the Petersen graph with parameters $(10,3,0,1)$. Strongly regular graphs are closely related to a second important combinatorial structure, a block design. Goethals and Seidel \cite{GS70} pointed out that strongly regular graphs can be derived from block designs. For an extensive overview on strongly regular graphs we refer to \cite{bvm}.
\par Introduced by Tits \cite{Tit57}, combinatorics on sets can be considered as the limit case $q$ tending to $1$ of combinatorics of vector spaces over the finite field with $q$ elements. Here, intersection of subsets corresponds to intersection of subspaces, but the union of subsets corresponds to the span of subspaces, as to preserve the similarities between the lattice of subsets of a set with $v$ elements and the lattice of subspaces of a $v$-dimensional vector space over $\F_q$. Naturally, cardinalities of subsets correspond to dimensions of subspaces.
\par A combinatorial object in the subspace lattice of $\F^{v}_{q}$ that behaves similarly (i.e. with respect to intersection and span/union) as an object in the subset lattice on a set of $v$ elements is called a $q$-analog of the latter. 
The interest in the concept of $q$-analogs of combinatorial objects has been renewed in the past decade due to the connection between $q$-analogs of block designs (also called designs over finite fields, see \cite{qsteiner, buratti,Tho87}), in particular Steiner systems, and error correction in random linear network coding (see \cite{mario-subspace, KK08}).

The goal of this article is to apply the concept of $q$-analogs to strongly regular graphs. We define $q$-analogs of strongly regular graphs (see Definition \ref{def:qgraph}) and give some examples. We derive some properties and show connections to classical strongly regular graphs and $q$-analogs of block designs, thereby demonstrating the similarities with classical strongly regular graphs. In Section \ref{sec:classification}, we also present a complete classification of strongly regular $q$-ary graphs.
\par The definitions of $q$-analogs of designs or the $q$-analogs of group divisible designs (recently introduced in \cite{qGDD}), as well as the definition of the $q$-analog of a strongly regular graph that we will introduce, have a condition on the intersection of the involved subspaces. But while the former specify the number of blocks through subspaces (of a given dimension), the latter will specify the size of the intersection. A connection between both is given in Theorem \ref{thm:design}.

\section{Preliminaries}

In this section, we recall some basic definitions and results for strongly regular graphs. For $0\le s\le v$, we denote by $\binom{V}{s}$ the set of all $s$-element subsets of a set $V$ with $v$ elements. Its cardinality is the binomial coefficient $\binom{v}{s}$.


\begin{defn}\label{def:graph}
	A simple (undirected) graph $\Gamma$ is a pair $(V,E)$ where $V$ is a set and $E\subseteq\binom{V}{2}$. The elements of $V$ and $E$ are called \emph{vertices} and \emph{edges}, respectively. If there is an edge $\{X,Y\}$, then we say that $X$ and $Y$ are \emph{adjacent}.
\end{defn}

\begin{rem}\label{rem:subsetsofsize1}
	Note that we can  identify the set of vertices of a graph with the set $\binom{V}{1}$ and express everything in terms of 1- and 2-subsets. Via this point of view we will define the $q$-analog, substituting subsets by subspaces and cardinalities by dimensions.
\end{rem}

\begin{defn}\label{def:nbh}
	For a graph $\Gamma$ a vertex $y$ is called a neighbor of a vertex $x$ if $x$ and $y$ are adjacent. The \emph{(closed) neighborhood} $N_{\Gamma}(x)$ of a vertex $x$ is the set of all its neighbors and $x$ itself.
\end{defn}

The open neighborhood $N_{\Gamma}(x)$ of a vertex $x$ in a graph $\Gamma$ is the set of its neighbors (so without $x$). Note that most authors use neighborhood (without specification) for the open neighborhood, but to stress the similarity with the $q$-analogs we will only consider closed neighborhoods here, and thus an unspecified neighborhood will always be closed in this article.

\begin{rem}\label{rem:neigborhood}
	Note that for an isolated vertex (a vertex without neighbors) of a graph $\Gamma=(V,E)$ we have $N_{\Gamma}(x)=x$. For all non-isolated vertices we have $N_{\Gamma}(x)=\cup_{x\in e\in E}\;e$.
\end{rem}

\begin{rem}\label{rem:graphasmap}
	Given the set $V$ of vertices of a graph, we can identify a graph $\Gamma$ with the map $\phi:V\to\mathcal{P}(V):x\mapsto N_{\Gamma}(x)$. This map has as properties that $x\in\phi(x)$ for any $x\in V$ and that $x\in\phi(y)\iff y\in\phi(x)$.
\end{rem}

\begin{defn}\label{def:regular}
	A graph is called \emph{$k$-regular} if each vertex has exactly $k$ neighbors, equivalently if each neighborhood has size $k+1$. The set of all $k$-regular graphs on $v$ vertices will be denoted by $\RG(v,k)$.
\end{defn}

\begin{rem}
	In the light of Remark \ref{rem:graphasmap} and the comment after Definition \ref{def:graph}, we can think of a $k$-regular graph as a map $\binom{V}{1}\to\binom{V}{k+1}$ for a set $V$. 
\end{rem}

\begin{defn}\label{def:srg}
	A $k$-regular graph on $v$ vertices such that any two adjacent vertices have $\lambda$ common neighbors and such that any two non-adjacent vertices have $\mu$ common neighbors is called strongly regular with parameters $(v,k,\lambda,\mu)$. The set of all strongly regular graphs with parameters $(v,k,\lambda,\mu)$ is denoted by $\SRG(v,k,\lambda,\mu)$.
\end{defn}

In terms of neighborhoods, the previous definition can be given as follows: a graph $\Gamma\in\RG(v,k)$ is in $\SRG(v,k,\lambda,\mu)$ if and only if for all vertices $x,y$ with $x\ne y$ we have that
\[
	|(N_{\Gamma}(x)\cap N_{\Gamma}(y))\setminus\{x,y\}|=\begin{cases}
	\lambda&\text{if $x$ and $y$ are adjacent,}\\
	\mu&\text{otherwise.}
	\end{cases}
\]

The standard result connecting the four parameters of a strongly regular graph is the following.

\begin{thm}[{\cite[Equation (2.5)]{Bos63}}]\label{thm:standardequation}
	If there is a strongly regular graph with parameters $(v,k,\lambda,\mu)$, then
	\[
	k(k-1-\lambda)=(v-k-1)\mu\:.
	\]
\end{thm}

We conclude this preliminary section by introducing block designs and their connection to strongly regular graphs.

\begin{defn}
A (block) design $(\mathcal{P},\mathcal{L})$ with parameters $t\text{-}(v,k,\lambda)$ consists of a set of \emph{points} $\mathcal{P}$ and a set of \emph{blocks} $\B\subseteq \binom{\mathcal{P}}{k}$ such that 
\[
\forall\, T\in\binom{\mathcal{P}}{t}:|\{K\in \B\mid T\subseteq K\}|=\lambda.
\]
\end{defn}

The following result connects a specific class of strongly regular graphs with symmetric 2-designs (\cite[p. 258]{haemers}) via neighborhoods. It is a so-called \emph{neighborhood design} (another can be made considering the open neighborhood).

\begin{thm}\label{thm:graphtodesign}
	If $\Gamma=(V,E)\in\SRG(v,k,\mu-2,\mu)$, then the set $\B_{\Gamma}:=\left\{N_{\Gamma}(x)\mid x\in V\right\}$ of all neighborhoods determines a symmetric $2\text{-}(v,k+1,\mu)$ design with point set $V$.
\end{thm}

\section{Strongly regular \texorpdfstring{$q$}{q}-ary graphs}
 
In this section, we will introduce and discuss (strongly regular) $q$-ary graphs. When defining the $q$-analog of a graph, we consider a vector space $V$ over a finite field $\F_{q}$, replacing the set in the graph definition. For a (classical) graph, the edges correspond to 2-subsets, so it is natural to define the edges of a $q$-ary graph as 2-dimensional subspaces of $V$. We mentioned in Remark \ref{rem:subsetsofsize1} that the vertices of a graph can be considered both as the elements of a set and as its 1-subsets. Looking at the former interpretation one could consider the vectors of $V$ as the vertices of the $q$-analog of a graph. However, in that case there is a unique vertex, namely the zero vector, which is automatically contained in every edge, a degenerate situation. So, it is natural to define a $q$-ary graph as follows. We use the notation $\gstxt{V}{s}{}$ for the set of $s$-dimensional subspaces of the vector space $V$.
\begin{defn}\label{def:qgraph}
	A $q$-ary graph $\E$ is a pair $(V,E)$ where $V=\gstxt{W}{1}{}$ and $E\subseteq\gstxt{W}{2}{}$, with $W$ a vector space over $\F_{q}$. The elements of $V$ and $E$ are called \emph{vertices} and \emph{edges}, respectively. If $\langle X,Y\rangle$ is an edge, then we say that the vertices $X$ and $Y$ are \emph{adjacent}.
\end{defn}

\begin{rem}\label{rem:projective}
	A $q$-ary graph is thus a hypergraph whose vertices are the vector lines of a vector space $V$ over $\F_{q}$, and whose edges correspond to vector planes of $V$. One can also consider this in the context of a projective geometry. Then the vertices of a $q$-ary graph are the points of the projective space $\PG(W)$ and the edges are lines of this projective space; the edge set is thus a subset of the line set.
\end{rem}

We now define the neighborhood similarly to how we define the neighborhood for classical graphs (Definition \ref{def:nbh}).

\begin{defn}
	For a vertex $X$ of a $q$-ary graph $\E$ we define its neighborhood $N_{\E}(X)$ as the set of all vertices adjacent to it, and itself.
\end{defn}

\begin{rem}
	We note the similarity with Remark \ref{rem:neigborhood}. For an isolated vertex $X$ (a vertex without neighbors) of a graph $\E=(V,E)$ we have $N_{\E}(X)=X$. For a non-isolated vertex $X$ we have $N_{\E}(X)=\cup_{X\in e\in E}\;e$.
\end{rem}

We now introduce regular $q$-ary graphs. Note the similarity with Definition \ref{def:regular}.

\begin{defn}
	A $q$-ary graph $\E$ is $k$-regular if for any vertex $X$ the neighborhood $N_{\E}(X)$ is a $(k+1)$-dimensional subspace. The set of all $k$-regular $q$-ary graphs with underlying vector space $\F_q^v$ will be denoted by $\RG(v,k;q)$.
\end{defn}

If all the neighborhoods are $(k+1)$-dimensional, we call a $q$-ary graph $k$-regular instead of $(k+1)$-regular, to preserve to analogy with $k$-regular graphs, where the neighborhood of each vertex has size $k+1$. Alternatively, we could say that for a $k$-regular graph $\Gamma$ we have $|N_{\Gamma}(x)\setminus\{x\}|=k$ for any vertex $x$, while for a $k$-regular $q$-ary graph $\E$ we have $\dim(N_{\E}(X)/X)=k$ for any vertex $X$.

\begin{rem}
	A $k$-regular $q$-ary graph $\E$ on a vector space $V$ can be identified with the map $\gstxt{V}{1}{}\to\gstxt{V}{k+1}{}:X\mapsto N_{\E}(X)$. This is a $q$-analog for the map described in Remark \ref{rem:graphasmap}.
\end{rem}

We now introduce strongly regular $q$-ary graphs, as a $q$-analog to the classical strongly regular graphs (see Definition \ref{def:srg}).

\begin{defn}
A regular $q$-ary graph $\E\in\RG(v,k;q)$ is a strongly regular $q$-ary graph with parameters $(v,k,\lambda,\mu;q)$ if and only if for all vertices $X,Y$, with $X\ne Y$, we have that
	\[
	|(N_{\E}(X)\cap N_{\E}(Y))\setminus\{X,Y\}|=\begin{cases}
	\lambda&\text{if $X$ and $Y$ are adjacent,}\\
	\mu&\text{otherwise.}
	\end{cases}
	\]
The set of all strongly regular $q$-ary graphs with these parameters $(v,k,\lambda,\mu)$ is denoted by $\SRG(v,k,\lambda,\mu;q)$.
\end{defn}

\begin{rem}
	A  $k$-regular $q$-ary graph $\E$ in $\F_q^v$ can be considered as a set of lines in $\PG(v-1,q)$, such that for any point $P$ the union of the lines from $\E$ that contain $P$ is a $k$-subspace of $\PG(v-1,q)$, see also Remark \ref{rem:projective}. A $k$-regular $q$-ary graph $\E$ in $\F_q^v$ is strongly regular with parameters $(v,k,\lambda,\mu)$ if for every two distinct points $Y$ and $Z$ the number of points $X$ such that both lines $XY$ and $XZ$ correspond to edges of $\E$ is equal to $\lambda$ if the line $YZ$ corresponds to an edge of $\E$, and $\mu$ otherwise.  
\end{rem}

We now prove an analog for Theorem \ref{thm:standardequation}, stressing the similarity between strongly regular graphs and strongly regular $q$-ary graphs. We use the notation 
\[
[a]_{q}=\sum_{i=0}^{a-1}q^i=\frac{q^{a}-1}{q-1}=\left|\gs{\F^{a}_{q}}{1}{}\right|.
\]
The number of vertices of a graph in $\RG(v,k;q)$ thus equals $[v]_{q}$.

\begin{thm}\label{thm:main1}
If there is a strongly regular $q$-ary graph with parameters $(v,k,\lambda,\mu;q)$, then
\[
\left([k+1]_q-1\right)\left([k+1]_q-2-\lambda\right)=\left([v]_q-[k+1]_q\right)\mu.
\]
\end{thm}
\begin{proof}
Let $\E=(V,E)\in\SRG(v,k,\lambda,\mu;q)$ be a $q$-ary graph and consider a vertex $P\in V$. We count the tuples $(X,Y)\in V\times V$ with $P\neq X\neq Y\neq P$, $\langle X,P\rangle,\langle X,Y\rangle\in E$ and $\langle P,Y\rangle\notin E$, in two ways.
\par On the one hand, there are $\left([k+1]_q-1\right)$ choices for $X$. Given $X$, there are $\lambda$ vertices that are adjacent to both $P$ and $X$, and so $X$ has $\left([k+1]_q-1\right)-1-\lambda$ neighbors that are different from $P$ and not adjacent with $P$.
\par On the other hand, there are $\left([v]_q-[k+1]_q\right)$ vertices that are not in the neighborhood of $P$, so this is the number of choices for $Y$. Given $Y$, there are $\mu$ vertices that are a neighbor of both $P$ and $Y$ by the definition of a strongly regular graph. The statement follows.
\end{proof}

\begin{rem}
By considering the limit case $q\to 1$ for the equation \[
([k+1]_q-1)([k+1]_q-2-\lambda)=([v]_q-[k+1]_q)\mu
\]
we obtain
\[
(k+1-1)(k+1-2-\lambda)=(v-(k+1))\mu
\]
which simplifies to the equation in Theorem \ref{thm:standardequation}.
\end{rem}

We now present some examples of strongly regular $q$-ary graphs.

\begin{ex}\label{ex:completegraph}
	For a vector space $V=\F_{q}^{v}$ the set $E=\gstxt{V}{2}{}$ contains all $2$-dimensional subspaces of $V$. This edge set determines a trivial regular $q$-ary graph $\E$ which we call the \emph{complete $q$-ary graph}. For any vertex $X\in\gstxt{V}{1}{}$ we find $N_{\E}\left(X\right)=V$. Hence, we have $\E\in\RG(v,v-1;q)$.
	\par Each distinct pair of $1$-dimensional subspaces of $V$ span an element of $E$. Moreover, they are both adjacent to the $[v]_{q}-2$ other vertices (all other 1-dimensional subspaces). Therefore, $\E\in\SRG(v,v-1,[v]_q-2,\mu;q)$ for any $\mu$. Note that $\mu$ is not defined for this $q$-ary graphs since there are no two non-adjacent 1-dimensional subspaces in $\E$.
	\par This $q$-ary graph is the $q$-analog of the complete graph on $v$ vertices, which has parameters $(v,v-1,v-2,\mu)$, $\mu$ arbitrary. Any graph can be considered as a subset of the complete graph, and similarly any $q$-ary graph can be considered as a subset of the complete $q$-ary graph.
\end{ex}

\begin{ex}
	Consider a vector space  $V=\F_{q}^{v}$ and a set $\mathcal{S}\subseteq\gstxt{V}{t}{}$ such that each 1-dimensional subspace is contained in precisely one element of $\mathcal{S}$, $t\geq2$. In a projective setting, this is called a $(t-1)$-spread of $\PG(v-1,q)$, and it is known to exist if and only if $t\mid v$ (see \cite{segre}). Note that $|\mathcal{S}|=\frac{q^{v}-1}{q^{t}-1}$. Consider now the set $E$ of all vector planes which are contained in an element of $\mathcal{S}$; it defines a $q$-ary graph $\E$. Then, $\E\in\SRG(v,t-1,[t]_{q}-2,0;q)$. It is the disjoint union of $\frac{q^{v}-1}{q^{t}-1}$ complete $q$-ary graphs on a $t$-dimensional vector space. This is a $q$-analog of a graph $\Gamma$ on $v$ vertices which is the disjoint union of $\frac{v}{t}$ complete graphs on $t$ vertices; we know $\Gamma\in\SRG(v,t-1,t-2,0)$. If $t=v$, then we find the complete ($q$-ary) graph.
	\par In particular, for $t=2$ and $v$ even, any graph in $\RG(v,1)$ is a disjoint union of $\frac{v}{2}$ edges, and thus $\RG(v,1)=\SRG(v,1,0,0)$. Similarly, the edge set of any $q$-ary graph $\E\in\RG(v,1;q)$ is actually a set of 2-dimensional subspaces such that each 1-dimensional subspace is contained in precisely one element of it, and as such $\RG(v,1;q)=\SRG(v,1,q-1,0;q)$.
\end{ex}

\begin{ex}\label{ex:sypmlectic}
	Let $\varphi$ be a symplectic polarity on the vector space $\F_{q}^{v}$, $v$ even, and let $\E$ be the $q$-ary graph with as edge set the totally isotropic vector planes of $\varphi$. The totally isotropic vector planes through a vector line form a hyperplane, so $\E\in\SRG(v,v-2,\mu-2,\mu;q)$ with $\mu=\gstxt{v-2}{1}{q}$.
\end{ex}

We now want to prove an analog for Theorem \ref{thm:graphtodesign}. To that end we recall the definition of a $q$-analog of a design. These so-called $q$-designs have been the subject of a lot of investigations in recent years (see Section \ref{sec:intro}).

\begin{defn}
	A $t\text{-}(v,k,\lambda;q)$ design $(\mathcal{P},\mathcal{B})$ consists of a set of \emph
	points $\mathcal{P}=\gstxt{\F_q^v}{1}{}$ and a set of \emph{blocks} $\B\subseteq\gstxt{\F_q^v}{k}{}$
	such that
	\[
	\forall\, T\in\gs{\F_q^v}{t}{}:|\{K\in\B\mid T\subseteq K\}|=\lambda.
	\]
\end{defn}

\begin{thm}\label{thm:design}
If $\E\in\SRG(v,k,\mu-2,\mu;q)$, with vertex set $V$, then the set
\[
\B_{\E}:=\left\{N_{\E}(X)\mid X\in V\right\}
\]
of all neighborhoods determines a $2\text{-}(v,k+1,\mu;q)$ design with point set $V$.
\end{thm}
\begin{proof}
It follows immediately from the definitions that the blocks are $(k+1)$-subspaces. Now, consider a 2-dimensional subspace $\sigma$ of the underlying vector space of $\E$. We need to check that there are $\mu$ elements of $\B_{\E}$ that contain $\sigma$. Consider two distinct 1-dimensional subspaces $X$ and $Y$ of $\sigma$ (vertices in the $q$-graph setting, points in the $q$-design setting). Any subspace $N_{\E}(Z)$ that contains $X$ and $Y$ contains $\sigma$, and vice versa. So,
\begin{align*}
	|\{B\in\B_{\E}\mid\sigma\subset B\}|&=|\{B\in\B_{\E}\mid X,Y\subset B\}|\\
	&=|\{Z\in V\mid X,Y\in N_{Z}(\E)\}|\\
	&=|\{Z\in V\mid Z\in N_{X}(\E)\wedge N_{Y}(\E)\}|\\
	&=|N_{\E}(X)\cap N_{\E}(Y)|\;.
\end{align*}
Since $\E\in\SRG(v,k,\mu-2,\mu;q)$, we know that $|N_{\E}(X)\cap N_{\E}(Y)|$ equals $(\mu-2)+2=\mu$ if $X$ and $Y$ are adjacent, and equals $\mu$ if $X$ and $Y$ are not adjacent.
\end{proof}

It is well known that a $2\text{-}(v,k,\lambda;q)$ $q$-ary design produces a $2\text{-}([v]_q,[k]_q,\lambda)$ design. Similarly, the following theorem holds.

\begin{thm}\label{thm:qsrg->srg}
If there exists a strongly regular $q$-ary graph with parameters $(v,k,\lambda,\mu;q)$, then there exists a strongly regular graph with parameters $([v]_q,[k+1]_q-1,\lambda,\mu)$.
\end{thm}
\begin{proof}
Let $\E\in\SRG(v,k,\lambda,\mu;q)$. Define a simple graph $\Gamma$ whose vertices are the vertices of the $q$-graph, i.e.~the 1-spaces of $\F_q^v$, where two vertices of $\Gamma$ are adjacent if the corresponding vertices in $\E$ are adjacent. It follows immediately from the definition of a strongly regular $q$-graph that $\Gamma$ is strongly regular with parameters $([v]_q,[k+1]_q-1,\lambda,\mu)$.
\end{proof}

\section{Classification of the strongly regular \texorpdfstring{$q$}{q}-ary graphs}\label{sec:classification}

We conclude this paper by classifying the strongly regular $q$-ary graphs. We first prove that there are three families of strongly regular $q$-ary graphs.

\begin{lemma}\label{lem:srgnumbers}
	If $\E\in\SRG(v,k,\lambda,\mu;q)$, then
	\begin{itemize}
		\item $\E$ is the disjoint union of $\frac{q^{v}-1}{q^{k+1}-1}$ complete graphs on the elements of a $(k+1)$-spread, or
		\item $v-2=k$, and $\lambda+2=\mu=\frac{q^{k}-1}{q-1}$, or
		\item $(v,k,\lambda,\mu)=(5,2,q-1,1)$.
	\end{itemize}
\end{lemma}
\begin{proof}
	The neighborhood of a vertex of $\E$ is a vector subspace of dimension $k+1$. Consequently, the common neighborhood of two (non-)adjacent vertices is also a subspace. Hence, there are integers $l,m$ such that $\lambda=\frac{q^{l}-1}{q-1}-2$ and $\mu=\frac{q^{m}-1}{q-1}$. From Theorem \ref{thm:main1}, we then derive that
	\begin{align}\label{eq1}
	&&\left(\frac{q^{k+1}-1}{q-1}-1\right)\left(\frac{q^{k+1}-1}{q-1}-\frac{q^{l}-1}{q-1}\right)&=\frac{q^{m}-1}{q-1}\left(\frac{q^{v}-1}{q-1}-\frac{q^{k+1}-1}{q-1}\right)\nonumber\\
	&\Leftrightarrow&q^{l+1}\left(q^{k}-1\right)\left(q^{k-l+1}-1\right)&=q^{k+1}\left(q^{m}-1\right)\left(q^{v-k-1}-1\right)\;.
	\end{align}
	Consequently, if both sides are different from zero, we have $k=l$. Note that $q^{k}-1>0$, so we have one of the three following cases.
	\begin{itemize}
		\item If $q^{k-l+1}-1=q^{m}-1=0$, then $\mu=0$ and $\lambda+2=\frac{q^{l}-1}{q-1}=\frac{q^{k+1}-1}{q-1}$, which means that any two adjacent vertices have exactly the same neighborhood, and two non-adjacent vertices have no common neighbors. Hence, $\F_{q}^{v}$ is partitioned in $\frac{q^{v}-1}{q^{k+1}-1}$ subspaces of dimension $k+1$, such that the edges of $\E$ are precisely the vector planes in these subspaces. I.e., $\E$ is the disjoint union of $\frac{q^{v}-1}{q^{k+1}-1}$ complete graphs on the elements of a $(k+1)$-spread of $\F_{q}^{v}$.
		\item If $q^{k-l+1}-1=q^{v-k-1}-1=0$, then $v=k+1=l$, so $\E$ is a complete graph. This is a particular instance of the first possibility in the statement of the theorem.
		\item If $k=l$, we have that $q^{k-l+1}-1=q-1>0$, and consequently $m\geq1$. Equation \eqref{eq1} reduces to
		\begin{align}\label{eq2}
		&&\left(q^{k}-1\right)\left(q-1\right)&=\left(q^{m}-1\right)\left(q^{v-k-1}-1\right)\nonumber\\
		&\Leftrightarrow&q^{k}-q^{k-1}-1&=q^{v-k-2+m}-q^{m-1}-q^{v-k-2}\;.
		\end{align}
		We again distinguish between two cases.
		\begin{itemize}
			\item If $m>1$, then reducing \eqref{eq2} modulo $q$, we see that $v-k-2=0$. From \eqref{eq2} it then follows that $q^{k-1}(q-1)=q^{m-1}(q-1)$, so $m=k$. We find the second possibility in the statement of the theorem.
			\item If $m=1$, then $q^{k-1}(q-1)=q^{v-k-2}(q-1)$, so $v=2k+1$. Clearly, then $\mu=1$. Let $P_{0},P_{1},P_{2}$ be vertices such that $P_{1}$ is adjacent with both $P_{0}$ and $P_{2}$, and $P_{0}$ and $P_{2}$ are non-adjacent; note that this is always possible. Let $\alpha_{i}=N_{\E}(P_{i})\in\gstxt{\F^{v}_{q}}{k+1}{}$, $i=0,1,2$. Since $k=l$, we have that $\dim(\alpha_{0}\cap\alpha_{1})=\dim(\alpha_{0}\cap\alpha_{2})=k$. However, then $\dim(\alpha_{0}\cap\alpha_{2})\geq k-1$, so $1=\mu\geq\frac{q^{k-1}-1}{q-1}$. So, $k\leq 2$. Since $\lambda\geq0$, we must have $k=2$ and we find the third possibility in the statement of the theorem.\qedhere
		\end{itemize}
	\end{itemize}
\end{proof}

\begin{lemma}\label{lem:symplectic}
	If $\E=(V,E)\in SRG(k+2,k,\mu-2,\mu;q)$ with $\mu=\frac{q^{k}-1}{q-1}$, then $\E$ is a $q$-ary graph described in Example \ref{ex:sypmlectic}.
\end{lemma}
\begin{proof}
	For each $X\in V$, its neighborhood $N_{\E}(X)$ is a hyperplane in $\F^{k+2}_{q}$. So, for $P\in V$ and $e\in E$ either $e$ is contained in $N_{\E}(P)$ or it shares a vector line with $N_{\E}(P)$. Hence, the point-line geometry $(V,E)$ fulfills the one-or-all axiom. Since no point is collinear to all other points, the geometry $(V,E)$ is a (non-degenerate) polar space by the Buekenhout-Shult axioms (see \cite{bs}). The corresponding polarity maps a vector line on the hyperplane that is its neighborhood. Since all vector lines are totally isotropic this polarity is necessarily symplectic. So, $\E$ indeed corresponds to Example \ref{ex:sypmlectic}.
\end{proof}

\begin{lemma}\label{lem:srg5}
	There are no strongly regular $q$-ary graphs with parameters $(5,2,q-1,1;q)$.
\end{lemma}
\begin{proof}
	Assume that $\E\in\SRG(5,2,q-1,1;q)$. Denote its set of vertices, i.e.~$\gstxt{\F^{5}_{q}}{1}{}$, by $\mathcal{P}$, and its set of edges by $\mathcal{L}$. Let $G$ be the simple, bipartite graph with vertex set $\mathcal{P}\cup\mathcal{L}$ where $X\in\mathcal{P}$ is adjacent with $Y\in\mathcal{L}$ if $X\subset Y$ (considered as vector subspaces). Since a vector plane contains $q+1$ vector lines, and since there are $q+1$ vector planes through a vector line in a 3-space, the graph $G$ is $(q+1)$-regular. It is immediate that $|\mathcal{P}|=\frac{q^{5}-1}{q-1}$, and by double counting the adjacent pairs $(X,Y)\in\mathcal{P}\times\mathcal{L}$, we get that $|\mathcal{L}|=\frac{q^{5}-1}{q-1}$.
	\par From $\lambda=q-1$ it follows that no two elements of $\mathcal{P}$ have the same neighborhood in $\E$. Since neighborhoods corresponds to 3-spaces in $\F^{5}_{q}$ this implies that there are no cycles of length six in $G$; note that there are no cycles of length 4 in $G$ since two vector lines determine a unique vector plane. If there would be a cycle of length 8 in $G$, then there are $X_{1},X_{2},X_{3},X_{4}\in\mathcal{P}$ such that $X_{i-1},X_{i+1}\subset N_{\E}(X_{i})$ and $X_{i+2}\not\subset N_{\E}(X_{i})$, where the indices are considered modulo 4. However, then $X_{2}$ and $X_{4}$ are both in $N_{\E}(X_{1})\cap N_{\E}(X_{3})$, contradicting $\mu=1$.
	\par Consider $X,X'\in\mathcal{P}$ with $X\neq X'$. If $X'$ is contained in $N_{\E}(X)$ then $d_{G}(X,X')=2$. If $X'\not\subset N_{\E}(X)$ then $d_{G}(X,X')=4$ since $\mu=1$ for $\E$. For any $X\in\mathcal{P}$ and $Y\in\mathcal{L}$ we get that
	\begin{itemize}
		\item $d(X,Y)=1$ iff $X\subset Y$,
		\item $d(X,Y)=3$ iff $X\not\subset Y$ and $Y$ has a non-trivial intersection with $N_{\E}(X)$, and
		\item $d(X,Y)=5$ iff $Y$ and $N_{\E}(X)$ intersect trivially.
	\end{itemize} 
	\par It follows that for a vertex $Y\in G$ that belongs to $\mathcal{L}$ there are $q+1$ vertices in $G$ at distance 1 and $q(q+1)$ vertices at distance 2. There are $q^{2}(q+1)$ vertices at distance 3 since there are no cycles of length 6, and $q^{3}(q+1)$ vertices at distance 4 since there are no cycles of length 8. We see that $1+q(q+1)+q^{3}(q+1)=|\mathcal{L}|$, so two vertices in $\mathcal{L}$ are at distance at most $4$ in $G$.
	\par We conclude that the diameter of $G$ is 5, and since there are no cycles of length 8, its girth must be 10. Consequently, $G$ is the incidence graph of a generalized pentagon with order $(q,q)$. However, it follows from the classical Feit-Higman result (see \cite{fh}) that such generalized polygons do not exist, contradicting the assumption.
\end{proof}

\begin{thm}\label{th:srgclass}
	If $\E\in\SRG(v,k,\lambda,\mu;q)$, then either $\E$ is the disjoint union of $\frac{q^{v}-1}{q^{k+1}-1}$ complete graphs on the elements of a $(k+1)$-spread, or $\E$ is a $q$-ary graph described in Example \ref{ex:sypmlectic}.
\end{thm}
\begin{proof}
	This follows immediately from Lemmas \ref{lem:srgnumbers}, \ref{lem:symplectic} and \ref{lem:srg5}.
\end{proof}

\section{Statements and Declarations}

\subsection{Declaration of competing interests}

The authors declare no conflict of interest.

\subsection{Data availability}

No data was used for the research described in the article.

\vspace*{0.5cm}

\noindent {\bf Acknowledgements} \\
D. Crnkovi\'c and V. Mikuli\'c Crnkovi\'c have been supported by Croatian Science Foundation under the project 6732 and M. De Boeck and A. \v Svob have been supported by Croatian Science Foundation under the project 5713.

\end{document}